\documentclass[letterpaper]{amsproc}
\usepackage{fullpage,amsmath,amsfonts}
\newtheorem{thm}{Theorem}
\newtheorem{prop}{Proposition}
\newtheorem{lem}{Lemma}
\newtheorem{cor}{Corollary}

\usepackage{young}

\begin{document}

\def\sh{\mbox{sh}}
\def\type{\mbox{\small type\,}}

\def\C{{\mathbb C}}
\def\N{{\mathbb N}}

\def\Z{{\mathbb Z}}
\def\R{{\mathbb R}}
\def\E{{\mathbb E}}
\def\P{{\mathbb P}}

\def\sgn{\mbox{sgn}}

\def\diag{\mbox{diag}}

\def\a{\alpha}
\def\n{\nu}
\def\m{\mu}

\def\la{{\lambda}}

\def\l{\lambda}
\def\g{\gamma}
\def\G{\Gamma}

\def\ep{\phi}

\def\tr{\mbox{{\rm tr} }}

\def\F{{\mathcal F}}
\def\Q{{\mathcal Q}}
\def\L{{\mathcal L}}

\def\T{{\mathcal T}}
\def\S{{\mathcal S}}

\title[Whittaker functions and related stochastic processes]{Whittaker functions and related stochastic processes}

\author{Neil O'Connell}
\address{Mathematics Institute, University of Warwick, Coventry CV4 7AL, UK}
\curraddr{}
\email{n.m.o-connell@warwick.ac.uk}
\thanks{}



\date{\today}

\dedicatory{}

\maketitle

\begin{abstract}  
We review some recent results on connections between Brownian motion,
Whittaker functions, random matrices and representation theory.
\end{abstract}

\section{Harish-Chandra formula, Duistermaat-Heckman measure and Gelfand-Tsetlin patterns}

Define $J_\lambda(x)=  h(\lambda)^{-1}\det(e^{\lambda_i x_j})$,
where $ h(\lambda)=\prod_{i<j}(\lambda_i-\lambda_j)$.  
For each $x$, $J_\lambda(x)$ is an analytic function of $\lambda$;
in particular, $J_0(x)=\left(\prod_{j=1}^{n-1} j!\right)  h(x)$.
The functions $J_\lambda(x)$ play a central role in random matrix theory.
For example, if $\Lambda$ and $X$ are Hermitian matrices with eigenvalues
given by $\lambda$ and $x$, respectively, then
\begin{equation}\label{iz}
\int_{U(n)}e^{\tr \Lambda U X U^*} dU=\frac{J_\lambda(x)}{J_0(x)},
\end{equation}
where the integral is with respect to normalised Haar measure on the unitary group.
This is known as the Harish-Chandra, or Itzykson-Zuber, formula.

Let $\beta=(\beta_t,\ t\ge 0)$ be a standard Brownian motion in $\R^n$ with drift $\lambda$.
Denote by $\P_x$ the law of $\beta$ started at $x$ and by $\E_x$ the corresponding
expectation.  Set $$\Omega =\{x\in\R^n:\ x_1>x_2>\cdots>x_n\},\qquad T=\inf\{t>0:\ \beta_t\notin\Omega\}.$$
For $\lambda,x\in\R^n$, write $\lambda(x)=\sum_i\lambda_i x_i$.  
\begin{prop} \label{sp}
For $x,\lambda\in\Omega$, $J_\lambda(x)=h(\lambda)^{-1} e^{\lambda(x)} \P_x(T=\infty)$.
\end{prop}
\begin{proof}
This is well known, see for example~\cite{bbo}.  
The function $u(x)=\P_x(T=\infty)$, $x\in\Omega$, satisfies
$\frac12 \Delta u + \lambda\cdot\nabla u =0$,
vanishes on the boundary of $\Omega$ and $\lim_{x\to\infty}u(x)=1$.
Here we write $x\to\infty$ to mean $x_i-x_{i+1}\to\infty$ for $i=1,\ldots,n-1$.
Hence $v(x)=e^{\lambda(x)}u(x)$ satisfies $ \Delta v=\sum_i\lambda_i^2v$,
vanishes on the boundary of $\Omega$ and
$\lim_{x\to\infty}e^{-\lambda(x)}v(x)=1$.  The function $\det(e^{\lambda_i x_j})$
also has these properties, so by uniqueness, 
$v(x)=\det(e^{\lambda_i x_j})$, as required.
\end{proof}

The Harish-Chandra formula
 has the following interpretation.  Pick $U$ at random according to the normalised
 Haar measure on $U(n)$ and let $\mu^x(dy)$
denote the law of the diagonal of the random matrix $UXU^*$.  Then the
integral becomes
$$\int_{U(n)}e^{\mbox{{\rm tr} }\Lambda U X U^*} dU=\int_{\R^n} e^{\lambda(y)} \mu^x(dy).$$
Setting $m^x(dy)=J_0(x)\mu^x(dy)$, we obtain
$$\int_{\R^n} e^{\lambda(y)} m^x(dy)=J_\lambda(x).$$
The measure $m^x$ is known as the Duistermaat-Heckman measure associated
with the point $x\in\Omega$.  It has the following properties, which are well-known.
The symmetric group $S_n$ acts naturally on $\R^n$ by permuting coordinates.  
The support of the measure $m^x$ is the convex hull of the set of images of $x$ 
under the action of $S_n$.  It has a piecewise polynomial density.  
This comes from the fact, which we will now explain, that the Duistermaat-Heckman
measure is the push-forward via an affine map of the Lebesgue measure on a 
higher dimensional polytope known as the Gelfand-Tsetlin polytope.

Let $x\in\Omega$ and denote by $GT(x)$ the polytope of Gelfand-Tsetlin
patterns with bottom row equal to $x$:
$$GT(x)=\{  P_{k,j},\ 1\le j\le k\le n:
 P_{k,j+1}\le P_{k-1,j}\le P_{k,j},\ 1\le j<k\le n,\ P_{n,\cdot}=x\}.$$
Define the {\em type} of a pattern $P$ to be the vector
\begin{equation}\label{type}
\type P = \left( P_{1,1}, P_{2,1}+P_{2,2}-P_{1,1} , \ldots , 
\sum_{j=1}^n P_{n,j}-\sum_{j=1}^{n-1} P_{n-1,j} \right) .
\end{equation}
Consider the map from $U(n)$ to $GT(x)$ defined by $U\mapsto P$ where: 
for each $1\le k\le n$,
$P_{k,\cdot}$ is the vector of eigenvalues of the $k^{th}$ principal minor of $UXU^*$.  
It is well-known (see, for example, \cite{bar} or \cite[Section 5.6]{ab} for a more general statement) 
that the push-forward of Haar measure under this map is the standard Euclidean measure on the 
polytope $GT(x)$.  Moreover, the diagonal of the matrix $UXU^*$ is equal to the 
type of the pattern $P$.  From this we obtain another integral representation for the
function $J_\lambda$ as
\begin{equation}\label{J}
J_\lambda(x)=\int_{GT(x)} e^{\lambda\cdot \type P} dP .
\end{equation}

\section{Whittaker functions}

Set
$H= \Delta -2\sum_{i=1}^{n-1}e^{-\alpha_i(x)}$,
where $\alpha_i=e_i-e_{i+1}$, $i=1,\ldots,n-1$.
Write $H=H^{(n)}$ for the moment; we will drop the superscript again later, whenever it is unnecessary.
For convenience we define $H^{(1)}=d^2/dx^2$ and $\psi^{(1)}_\lambda(x)=e^{\lambda x}$.
Following~\cite{gklo}, for $n\ge 2$ and $\theta\in\C$, define a kernel on 
$\R^ n\times\R^{ n-1}$ by
$$Q^{( n)}_\theta(x,y) = \exp\left( \theta
\left( \sum_{i=1}^ n x_i- \sum_{i=1}^{ n-1} y_i \right)
-\sum_{i=1}^{ n-1} \left( e^{y_i-x_i} + e^{x_{i+1}-y_i}\right) \right) .$$
Denote the corresponding integral operator by $\Q^{( n)}_\theta$,
defined on a suitable class of functions by
$$\Q^{( n)}_\theta f(x)=\int_{\R^{ n-1}} Q^{( n)}_\theta(x,y) f(y) dy.$$
The Whittaker functions $\psi^{(n)}_\lambda ,\lambda\in\C^n$ are defined recursively by
\begin{equation}\label{ior}
\psi^{( n)}_{\l_1,\ldots,\l_ n} = \Q^{( n)}_{\l_ n} \psi^{( n-1)}_{\l_1,\ldots,\l_{ n-1}}.
\end{equation}
As observed in~\cite{gklo}, the following intertwining relation holds:
\begin{equation}\label{i1}
(H^{( n)}-\theta^2)\circ \Q^{( n)}_\theta = \Q^{( n)}_\theta \circ H^{( n-1)}.
\end{equation}
This follows from the identity
$(H^{( n)}_x-\theta^2) Q^{( n)}_\theta(x,y) = H^{( n-1)}_y Q^{( n)}_\theta(x,y)$,
which is readily verified.  Combining (\ref{ior}) with the intertwining relation (\ref{i1}) 
yields the eigenvalue equation:
\begin{equation}\label{ee}
H^{( n)} \psi^{( n)}_\l = \left( \sum_{i=1}^ n \l_i^2\right)  \psi^{( n)}_\l.
\end{equation}
Let us now drop the superscripts and write $H=H^{(n)}$, $\psi_\l=\psi^{(n)}_\l$. 
Iterating (\ref{ior}) gives the following integral formula, due to Givental~\cite{givental,jk,gklo}
\begin{equation}\label{g}
\psi_\l(x) = \int_{\Gamma(x)} e^{\F_\l(  T)}  \prod_{k=1}^{n-1}\prod_{i=1}^k d  T_{k,i} ,
\end{equation}
where $\Gamma(x)$ denotes the set of real triangular arrays 
$(  T_{k,i},\ 1\le i\le k\le n)$ with $  T_{n,i}=x_i$, $1\le i\le n$,
and
$$\F_\l(  T)= \sum_{k=1}^n \l_k 
\left( \sum_{i=1}^k   T_{k,i}- \sum_{i=1}^{k-1}   T_{k-1,i} \right)
-\sum_{k=1}^{n-1}\sum_{i=1}^k \left( e^{  T_{k,i}-  T_{k+1,i}} + e^{  T_{k+1,i+1}-  T_{k,i}}\right).$$

Now, it is shown in~\cite{boc} that, for each $\lambda\in\Omega$, the equation
$H f=\sum_i\lambda_i^2 f$ has a unique solution $f=f_\lambda$ such that $e^{-\lambda(x)}f_\lambda(x)$ 
is bounded and $\lim_{x\to+\infty} e^{-\lambda(x)}f_\lambda(x)=1$, where we write $x\to+\infty$ to mean
$\alpha_i(x)=x_i-x_{i+1}\to+\infty$ for each $i$.  Moreover, by Feynman-Kac,
\begin{equation}\label{fk-f}
f_\lambda(x)=e^{\lambda(x)} \E_x \exp\left(-\sum_{i=1}^{n-1}\int_0^\infty e^{-\alpha_i(\beta_s)} ds\right),
\end{equation}
where $\beta_s$ is a Brownian motion in $\R^n$ with drift $\lambda$ as in the previous section.
The relation between the functions $f_\lambda$ and the Whittaker functions $\psi_\lambda$ is thus determined by
the following proposition.
\begin{prop} For $\lambda\in\Omega$,
\begin{equation}\label{lim}
\lim_{x\to+\infty} e^{-\l(x)} \psi_\l (x) =  \prod_{i<j} \G(\l_i-\l_j).
\end{equation}
\end{prop}
\begin{proof} We prove this by induction on $n$ using the recursion (\ref{ior}). 
Write $\psi_\lambda=\psi_\lambda^{(n)}$ as before, setting $\psi_\lambda^{(1)}(x)=e^{\lambda x}$.
Then $e^{-\lambda(x)}\psi^{(1)}_\lambda(x) =1$ and, for $n\ge 2$,
\begin{align*}
e^{-\lambda(x)}\psi^{(n)}_\lambda(x) =
\int_{\R^{n-1}} & \exp\left(-\sum_{i=1}^n \lambda_i x_i + \lambda_n \left(\sum_{i=1}^ n x_i- \sum_{i=1}^{ n-1} y_i \right) 
-\sum_{i=1}^{ n-1} \left( e^{y_i-x_i} + e^{x_{i+1}-y_i}\right) \right)\\
& \qquad \times  \psi^{(n-1)}_{\l_1,\ldots,\l_{n-1}}(y_1,\ldots,y_{n-1}) dy_1\ldots dy_{n-1}\\
= \int_{\R^{n-1}} & e^{\sum_{i=1}^{n-1} (\l_i-\l_n)y_i} \exp\left( -\sum_{i=1}^{n-1} e^{y_i} - \sum_{i=1}^{n-1} e^{x_{i+1}-x_i-y_i}\right)\\
& \times e^{-\sum_{i=1}^{n-1} \l_i(x_i+y_i)} \psi^{(n-1)}_{\l_1,\ldots,\l_{n-1}}(x_1+y_1,\ldots,x_{n-1}+y_{n-1}) dy_1\ldots dy_{n-1} .\\
\end{align*}
By induction, we immediately conclude that, for each $n$, if $x,\lambda\in\Omega$ then 
$e^{-\lambda(x)}\psi^{(n)}_\lambda(x)\le  \prod_{i<j} \G(\l_i-\l_j)$.  Here we are using
\begin{align*}
\int_{\R^{n-1}} e^{\sum_{i=1}^{n-1} (\l_i-\l_n)y_i} & \exp\left( -\sum_{i=1}^{n-1} e^{y_i} - \sum_{i=1}^{n-1} e^{x_{i+1}-x_i-y_i}\right)
dy_1\ldots dy_{n-1} \\
& \le 
\int_{\R^{n-1}} e^{\sum_{i=1}^{n-1} (\l_i-\l_n)y_i} \exp\left( -\sum_{i=1}^{n-1} e^{y_i} \right)
dy_1\ldots dy_{n-1}  = \prod_{i=1}^{n-1} \Gamma(\lambda_i-\lambda_n).
\end{align*}
It follows, again by induction
and using the dominated convergence theorem, that (\ref{lim}) holds for $\l\in\Omega$.
\end{proof}

\begin{cor} For $\lambda\in\Omega$,
\begin{equation}\label{fk}
\psi_\lambda(x)=\prod_{i<j} \G(\lambda_i-\lambda_j) e^{\lambda(x)}
\E_x \exp\left(-\sum_{i=1}^{n-1}\int_0^\infty e^{-\alpha_i(\beta_s)} ds\right).
\end{equation}
\end{cor}

\begin{cor} For $x,\lambda \in\Omega$,
$$J_\lambda(x)=\lim_{\beta\to\infty} \beta^{-n(n-1)/2}\psi_{\lambda/\beta}(\beta x).$$
\end{cor}
\begin{proof}
By Proposition~\ref{sp}, the statement is equivalent to
$$\lim_{\beta\to\infty} \beta^{-n(n-1)/2}\psi_{\lambda/\beta}(\beta x)
=  h(\lambda)^{-1} e^{\lambda(x)} \P_x(T=\infty).$$
This follows directly from \eqref{fk} by Brownian re-scaling.
\end{proof}

As shown in~\cite{boc}, the function $\psi_\lambda$, which can be defined by \eqref{fk}, is a class-one Whittaker function, 
as defined by Jacquet~\cite{ja} and Hashizume~\cite{h}.  In the notation of the paper~\cite{boc} we are 
taking $\Pi= \{\a_i/2,\ i=1,\ldots,n-1\}$, $m(2\alpha)=0$, $|\eta_\alpha|^2=1$ and 
$\psi_\nu(x)=2^q k_\nu(x)$ where $q=n(n-1)/2$.  In the paper~\cite{GLO}, the relationship between Givental 
integral formula and a recursive integral formula due to Stade~\cite{st} based on Jacquet's definition 
(see also \cite{is}) is described.

Givental's integral formula (\ref{g}) has a very similar structure to the formula~\eqref{J}.
Indeed, if we define the type of an array $(  T_{k,i},\ 1\le i\le k\le n)$ to be the vector
$$\type T = \left(T_{1,1}, T_{2,1}+T_{2,2}-T_{1,1} , \ldots , 
\sum_{j=1}^n T_{n,j}-\sum_{j=1}^{n-1} T_{n-1,j} \right) ,$$
and a measure
$$g(dT)=\prod_{k=1}^{n-1}\prod_{i=1}^k e^{-e^{  T_{k,i}-  T_{k+1,i}}}
e^{- e^{  T_{k+1,i+1}-  T_{k,i}}} d  T_{k,i}  = e^{\F_0(T)} \prod_{k=1}^{n-1}\prod_{i=1}^k d  T_{k,i} ,$$
then
$$\psi_\l(x) = \int_{\Gamma(x)} e^{\l\cdot \type T} g(dT).$$
On the other hand, if we replace the functions $e^{-e^{x-y}}$ in the reference
measure $g$ by the indicator functions $1_{x<y}$ to get a new reference measure
$$g_0(dT)=\prod_{k=1}^{n-1}\prod_{i=1}^k 1_{ T_{k,i}<  T_{k+1,i}}
1_{T_{k+1,i+1}< T_{k,i}} ,$$
then \eqref{J} can be written as
$$J_\l(x) = \int_{\Gamma(x)} e^{\l\cdot \type T} g_0(dT).$$

We note the following.
If $\lambda\in\iota\R^ n$ then $\overline{\psi_\lambda(x)}=\psi_{-\lambda}(x)$;
if $\lambda\in\iota\R^ n$ and $\nu\in\R^ n$, then $|\psi_{\lambda+\nu}(x)|\le \psi_\nu(x)$.
For each $x\in\R^ n$, $\psi_\lambda(x)$ is an entire, symmetric function of 
$\lambda\in\C^ n$~\cite{GLO,h,kl}.  There is a Plancherel 
theorem~\cite{wallach,an,GLO,kl} which states that the integral transform
\begin{equation}\label{u}
\hat f(\lambda)=\int_{\R^n} f(x)\psi_\lambda(x)dx
\end{equation}
is an isometry from $L_2(\R^n,dx)$ onto $L^{sym}_2(\iota\R^n,s_n(\lambda)d\lambda)$,
where $L_2^{sym}$ is the space of $L_2$
functions which are symmetric in their variables, $\iota=\sqrt{-1}$ and 
 $s_n(\lambda)d\l$ is the {\em Sklyanin measure} defined by
\begin{equation}\label{sklyanin}
s_n(\l)=\frac1{(2\pi \iota)^n n!} \prod_{j\ne k} \Gamma(\l_j-\l_k)^{-1}.
\end{equation}

For $x,\mu\in\R^n$, denote by $\sigma^x_\mu$ the probability measure on 
the set of real triangular arrays $(T_{k,i})_{1\le i\le k\le n}$ defined by 
$$\int f d\sigma^x_\mu = \psi_\mu(x)^{-1} 
 \int_{\Gamma(x)} f(T) 
e^{\F_\mu(T)}  \prod_{k=1}^{n-1}\prod_{i=1}^k dT_{k,i}.$$

Define a probability measure $\gamma^x_\mu$ by
 $$\int_{\R^n} e^{\lambda\cdot y} \gamma^x_\mu(dy) = \frac{\psi_{\mu+\lambda}(x)}{\psi_\mu(x)},
 \qquad \lambda\in\C^n.$$
 The probability measure $\gamma^x=\gamma^x_0$ is the analogue of the 
 (normalised) Duistermaat-Heckman measure in this setting. 
 The integral operator $K$ with kernel 
 $K(x,dy)=\psi_0(x) \gamma^x(dy)$ satisfies the intertwining relation $HK=K\Delta$.
 We can write $K(x,dy)= k(x,y) \rho_x(dy)$, where $k$ is a smooth kernel from 
 $\R^n$ to $\R^n_x=\{y\in\R^n:\ \sum_i y_i = \sum_i x_i\}$ and $\rho_x$ denotes the
 Euclidean measure on $\R^n_x$.  
 For $n=2$, $$k(x,y)=\exp(-e^{x_2-y_1}-e^{y_1-x_1})$$ and, for $n=3$, 
$$k(x,y)=\psi_0^{(2)}(a,b)=2K_0(2e^{(b-a)/2})$$
 where $$e^{-a}=e^{x_3-y_1-y_2}+e^{-x_1},\qquad e^b=e^{y_1}+e^{y_2}+e^{y_1+y_2-x_2}+e^{x_2},$$
 and $K_0$ denotes the Macdonald function with index $0$.

\section{Interpretation of $\gamma^x$ in terms of Brownian motion}

A reduced decomposition of an element $ w\in S_n$ is a minimal
expression of $w$ as a product of adjacent transpositions, that is,
$ w=s_{i_1}\ldots s_{i_r}$, where $s_i$ denotes the transposition 
$(i,i+1)$.  We will also refer to the word ${\bf i}=i_1i_2\ldots i_r$ as a reduced
decomposition.  By definition, any reduced decomposition has the same 
length $l(w)$, defined to be the length of $ w$.  There is a unique longest element
in $S_n$, namely the permutation
$$ w_0=\left(\begin{array}{cccc} 1 & 2 & \cdots & N \\ N& N-1 & \cdots & 1\end{array}\right) .$$
Its length is $n(n-1)/2$, as can be seen by taking the reduced decomposition
$${\bf i}=1\ 21\ 321\ \ldots\ n\ n-1\ \ldots 21.$$
The symmetric group acts on $\R^n$ by permutation of coordinates, and as
such is an example of a finite reflection group.  It is generated by the hyperplane
reflections $s_i=s_{\alpha_i}$, $i=1,\ldots,n-1$, defined for $x\in\R^n$ by
 $$s_i x=x-\alpha_i(x)\alpha_i,$$
 where $\alpha_i=e_i-e_{i+1}$.  Note that $s_i$ corresponds to the 
adjacent transposition $(i,i+1)$.

For a continuous path $\eta:(0,\infty)\to\R^n$, define $T_i=T_{\alpha_i}$ by
$$ T_i \eta(t)=\eta(t)+\left(\log\int_0^t e^{-\alpha_i(\eta(s))} ds\right) 
\alpha_i,\qquad  t> 0.$$  
Let $w=s_{i_1}\cdots s_{i_r}$ be a reduced decomposition.  
Then $$ T_w:=  T_{i_r}\cdots   T_{i_1}$$
depends only on $w$, not on the chosen decomposition~\cite{bbo}.

We now introduce a probability measure $\mathbb P$ under which $\eta$ 
is a Brownian motion in $\R^n$ with a drift $\mu$ and $\eta(0)=0$.  
In this setting, a very special role is played by the transform $T^{(n)}=T_{w_0}$.  
In the following we use the fact that this is well-defined for each $n$.
Write $\eta=(\eta^1,\ldots,\eta^n)$.
For each $k\le n$, set
$$( T_{k,1},\ldots, T_{k,k})= T^{(k)}(\eta^1,\ldots,\eta^k).$$

The evolution of the triangular array $   T_{k,j},\ 1\le j\le k \le n$
is given recursively as follows: $d   T_{1,1}=d\eta^1$ and, for $k\ge 2$,
\begin{align}\label{T-ev}
d   T_{k,1} &= d   T_{k-1,1} + e^{   T_{k,2}-   T_{k-1,1}} dt \nonumber \\
d   T_{k,2} &= d   T_{k-1,2} + \left( e^{   T_{k,3}-   T_{k-1,2}}-e^{   T_{k,2}-   T_{k-1,1}} \right) dt \nonumber \\
&\ \vdots \nonumber \\
d   T_{k,k-1} &= d   T_{k-1,k-1} + \left( e^{   T_{k,k}-   T_{k-1,k-1}}-e^{   T_{k,k-1}-   T_{k-1,k-2}} \right) dt \nonumber \\
d   T_{k,k} &= d\eta^k - e^{   T_{k,k}-   T_{k-1,k-1}} dt .
\end{align}
The process, which is clearly Markov, contains a number of projections which
are also Markov.  For example, setting $\xi_k=T_{k,k}$, we have, for $k\le n$,
$$d   \xi_k = d\eta^k - e^{  \xi_k-   \xi_{k-1}} dt .$$
This defines a simple interacting particle system on the real line, which has
very nice properties.  For example, in the coordinates $\sum_i\xi_i$ and
$\xi_{i+1}-\xi_{i}$, $1\le i\le n-1$, it has a product form invariant measure,
that is, a product measure which is invariant.

A remarkable fact is that each row in the pattern $T_{k,j}$ is a Markov process
with respect to its own filtration.  
This gives an interpretation of the measures $\gamma^x_\mu$ and $\sigma^x_\mu$
defined in the previous section.
\begin{thm}~\cite{noc}\label{main}
$T_{w_0}\eta(t),\ t>0$ is a diffusion process in $\R^n$ with infinitesimal generator 
$$\L_\mu = \frac12 \psi_\mu^{-1} \left(H-\sum_{i=1}^n\mu_i^2\right) \psi_\mu 
= \frac12 \Delta + \nabla\log \psi_\mu\cdot\nabla .$$
For each $t>0$, the conditional law of $\{T_{k,j}(t),\ 1\le j\le k\le n\}$,
given $\{T_{w_0}\eta(s),\ s\le t;\ T_{w_0}\eta(t)=x\}$, is $\sigma^x_\mu$
and the conditional law of $\eta(t)$, given
$\{T_{w_0}\eta(s),\ s\le t;\ T_{w_0}\eta(t)=x\}$, is $\gamma^x_\mu$.
The law of $T_{w_0}\eta(t)$ is given by
$$\nu^{\mu}_t(dx)=e^{-\sum_i \mu_i^2 t/2}  \psi_\mu(x)\theta_t(x)dx,$$ 
where 
\begin{equation}\label{theta}
\theta_t(x)=\int_{\iota \R^n}  \psi_{-\l} (x) e^{\sum_i \l_i^2 t/2} s_n(\l) d\l .
\end{equation}
\end{thm}
In the case $n=2$, this is equivalent to a theorem of Matsumoto and Yor~\cite{my}.

Write $\L=\L_0$ and $\nu_t=\nu^0_t$.  The diffusion with generator $\L$ is the
analogue of Dyson's Brownian motion in this setting and the measures $\nu_t$
and $\theta_t$ (the latter requires normalization) are the analogues of the Gaussian
unitary and Gaussian orthogonal ensembles, respectively.
The diffusion with generator $\L_\mu$ was introduced in~\cite{boc}.  When 
$\mu\in\overline{\Omega}$, it can be interpreted as a Brownian motion in $\R^n$
killed according to the potential $\sum_i e^{x_{i+1}-x_i}$ and then
conditioned to survive forever~\cite{k1,k2}.  The path-transformation $T_{w_0}$
is closely related to the geometric (lifting of the) RSK correspondence introduced by 
A.N. Kirillov~\cite{kirillov} and studied further by Noumi and Yamada~\cite{ny}.  
A discrete-time version of the above theorem, which works directly in the setting 
of the geometric RSK correspondence, is given in~\cite{cosz}.  
In the discrete-time setting the Whittaker functions continue to play a central role. 
See also~\cite{bc,bcr,ch,gnss,osz,ow} for further related developments.  

\section{Application to random polymers}

The following model was introduced in~\cite{oy2}.
The environment is given by a sequence $B_1,B_2,\ldots$ independent standard 1-dim Brownian motions.
For up/right paths $\phi \equiv \{0<t_1<\ldots < t_{N-1}< t\}$ (as shown in Figure 1), define
$$E(\phi)=B_1(t_1)+B_2(t_2)-B_2(t_1)+\cdots+B_N(t)-B_N(t_{N-1}),$$
$$P(d\phi)=Z^n_t(\beta)^{-1} e^{\beta E(\phi)} d\phi,
\qquad Z^n_t(\beta)= \int e^{\beta E(\phi)} d\phi .$$

\begin{figure}
\setlength{\unitlength}{.5cm}
\begin{small}

\begin{center}

\begin{picture}(12,7)(-1,-1)

\put(0,0){\vector(1,0){10}}
\put(0,0){\vector(0,1){6}}

\put(9,-1){$t$}
\put(7,-1){$t_{N-1}$}
\put(5.5,-1){$t_{N-2}$}
\put(4,-1){$t_3$}
\put(2.5,-1){$t_2$}
\put(1.5,-1){$t_1$}

\put(-1,-.2){$1$}
\put(-1,.8){$2$}
\put(-1,1.8){$3$}
\put(-1.5,3.8){$N-1$}
\put(-1,4.8){$N$}

\put(1.5,0){\dashbox{.1}(0,1)}
\put(2.5,1){\dashbox{.1}(0,1)}
\put(4,2){\dashbox{.1}(0,1)}
\put(5.5,3.5){\dashbox{.1}(0,.5)}
\put(7,4){\dashbox{.1}(0,1)}

\put(1.5,-.2){\line(0,1){.4}}
\put(2.5,-.2){\line(0,1){.4}}
\put(4,-.2){\line(0,1){.4}}
\put(5.5,-.2){\line(0,1){.4}}
\put(7,-.2){\line(0,1){.4}}
\put(9,-.2){\line(0,1){.4}}

\put(-.2,0){\line(1,0){.4}}
\put(-.2,1){\line(1,0){.4}}
\put(-.2,2){\line(1,0){.4}}
\put(-.2,4){\line(1,0){.4}}
\put(-.2,5){\line(1,0){.4}}

\linethickness{.7mm}

\put(0,0){\line(1,0){1.5}}
\put(1.5,1){\line(1,0){1}}
\put(2.5,2){\line(1,0){1.5}}
\put(4,3){\line(1,0){.3}}
\put(5.5,4){\line(1,0){1.5}}
\put(7,5){\line(1,0){2}}

\end{picture}
\end{center}

\caption{An up/right path $\phi \equiv \{0<t_1<\ldots < t_{n-1}< t\}$.}

\end{small}
\end{figure}
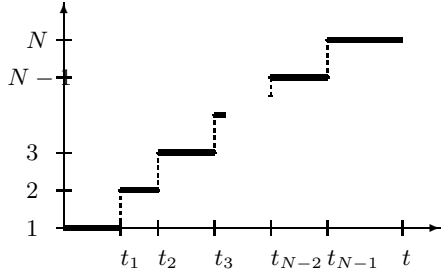

Set $X^n_1(t)=\log Z^n_t$ and, for $k=2,\ldots,n$,
$$X^n_1(t)+\cdots+X^n_k(t)=\log\int e^{E(\phi_1)+\cdots+E(\phi_k)}d\phi_1\ldots d\phi_k,$$
where the integral is over non-intersecting paths $\phi_1,\ldots,\phi_k$ from
$(0,1),\ldots,(0,k)$ to $(t,n-k+1),\ldots,(t,n)$.

Let $\eta=(B_n,\ldots,B_1)$.  Then $X=T_{w_0}\eta$ and the following holds.

\begin{thm}\cite{noc}\label{qtl}
The process $X(t), t>0$ is a diffusion in $\R^n$ with infinitesimal generator $\L$.
The distribution of $X(t)$ is given by $\nu_t$.  For $s>0$,
$$Ee^{- s Z^n_t } =
  \int s^{-\sum \l_i} \prod_{i}\Gamma(\l_i)^n e^{\frac12\sum_i\l_i^2 t} s_n(\l) d\l ,$$
where the integral is along (upwards) vertical lines with $\Re\l_i>0$ for all $i$.
\end{thm}
The free energy for this model is given by~\cite{oy2,moc}
$$\lim_{n\to\infty} \frac1n\log Z^n_n= \inf_{t>0}[ t - \Psi(t)],$$
almost surely, where $\Psi(z)=\Gamma'(z)/\Gamma(z)$.
The conjectured KPZ scaling behaviour for the fluctuations of $\log Z^n_n$
was (essentially) established by Sepp\"al\"ainen and Valk\'o~\cite{sv}; more recently,
Borodin, Corwin and Ferrari~\cite{bc,bcf} have proved the full KPZ universality 
conjecture for this model, namely that $\log Z^n_n$, suitably centered
and rescaled, converges in law to the Tracy-Widom $F_2$ distribution of 
random matrix theory.  See also~\cite{spohn}.

\section{Reduced double Bruhat cells and their parameterisations}

 The Weyl group associated with $GL(n)$
is the symmetric group $S_n$.  Each element $w\in S_n$ has a representative
$\bar w\in GL(n)$ defined as follows.  Denote the standard generators for
$\mathfrak{gl}_n$ by $h_i$, $e_i$ and $f_i$.  For example, for $n=3$,
$$h_1=\begin{pmatrix}
1& 0 &0\\
0 & 0 & 0 \\
0 & 0 & 0 \end{pmatrix},\qquad 
h_2 =\begin{pmatrix}
0& 0 &0\\
0 & 1 & 0 \\
0 & 0 & 0 \end{pmatrix},\qquad 
h_3 = \begin{pmatrix}
0& 0 &0\\
0 & 0 & 0 \\
0 & 0 & 1 \end{pmatrix},
$$
$$e_1=\begin{pmatrix}
0& 1 &0\\
0 & 0 & 0 \\
0 & 0 & 0 \end{pmatrix},\quad 
e_2 =\begin{pmatrix}
0& 0 &0\\
0 & 0 & 1 \\
0 & 0 & 0 \end{pmatrix},\quad f_1=\begin{pmatrix}
0& 0 &0\\
1 & 0 & 0 \\
0 & 0 & 0 \end{pmatrix},\quad 
f_2 =\begin{pmatrix}
0& 0 &0\\
0 & 0 & 0 \\
0 & 1 & 0 \end{pmatrix}.
$$

For adjacent transpositions $s_i=(i,i+1)$, define 
$$\bar{s}_i=\exp(-e_i)\exp(f_i)\exp(-e_i)=(I-e_i)(I+f_i)(I-e_i).$$  
In other words, 
$\bar s_i=\varphi_i\begin{pmatrix}0&-1\\1&0\end{pmatrix}$
where $\varphi_i$ is the natural embedding of $SL(2)$ into $GL(n)$
given by $h_i$, $e_i$ and $f_i$.  For example, when $n=3$,
$$\bar s_1 = \begin{pmatrix}
0& -1 &0\\
1 & 0 & 0 \\
0 & 0 & 1 \end{pmatrix},\qquad 
\bar s_2 =
\begin{pmatrix}
1& 0 &0\\
0 & 0 & -1 \\
0 & 1 &0 \end{pmatrix}.$$
Now let $w=s_{i_1}\ldots s_{i_r}$ be a reduced decomposition and define
$\bar{w}=\bar s_{i_1}\ldots \bar s_{i_r}$.  Note that $\overline{uv}=\bar u\bar v$
whenever $l(uv)=l(u)+l(v)$.
For $n=2$, $ w_0=s_1$ and $$\bar  w_0=\bar s_1 =  \begin{pmatrix}
0 & -1  \\
1 & 0 \end{pmatrix} .$$
For $n=3$, $ w_0=s_1s_2s_1=s_2s_1s_2$ is represented by
$$
\bar  w_0= \bar s_1\bar s_2\bar s_1=\bar s_2\bar s_1\bar s_2 = \begin{pmatrix}
0& 0 &1\\
0 & -1 & 0 \\
1 & 0 & 0 \end{pmatrix}  .$$

Denote the upper (respectively lower) triangular matrices in $GL(n)$
by $B$ and $B_-$, and the upper (respectively lower) uni-triangular matrices in $GL(n)$
by $N$ and $N_-$.
The group $GL(n)$ has two Bruhat decompositions
$$GL(n)=\bigcup_{u\in S_n} B\bar{u}B=\bigcup_{v\in S_n} B_{-}\bar{v}B_{-}.$$
The double Bruhat cells $G^{u,v}$ are defined, for $u,v\in S_n$, by 
$$G^{u,v}=B\bar{u}B\cap B_{-}\bar{v}B_{-}.$$
The reduced double Bruhat cells $L^{u,v}$ are defined by
$$L^{u,v}=N\bar{u}N\cap B_{-}\bar v B_{-}.$$
We also define the opposite reduced double Bruhat cells 
$M^{u,v}$ by
$$M^{u,v}=B\bar{u}B\cap N_{-}\bar v N_{-}.$$

The reduced double Bruhat cell $L^{w,e}$ (where $e$ denotes the identity in $S_n$) admits the following
parameterisations, one for each reduced decomposition of $w$. 
See~\cite{l,bz,bfz,fz}. Set 
$$Y_i(u)=\varphi_i\begin{pmatrix}u&0\\1&u^{-1}\end{pmatrix},\qquad i=1,\ldots,n-1.$$
Then, for any reduced decomposition ${\bf i}=i_1\ldots i_r$ of $w$, 
the map
$$(u_1,\ldots,u_r) \mapsto Y_{i_1}(u_1)\cdots Y_{i_r}(u_r)$$
defines a bijection between $\C_{\ne 0}^r$ and $L^{w,e}$. 
This bijection has the property that the totally positive part $L^{u,v}_{>0}$
of $L^{u,v}$ corresponds precisely to the subset $\R_{>0}^r$ of $\C_{\ne 0}^r$.
There are explicit transition maps which relate the parameters 
$(u_1,\ldots,u_r)$ corresponding to different reduced decompositions of $w$.

In the case $n=3$, the two representations of an element in $L^{w_0,e}$
corresponding to the words $121$ and $212$, denoting
the corresponding parameters by $(u_1,u_2,u_3)$ and $(u'_1,u'_2,u'_3)$,
respectively, are given by
$$
\begin{pmatrix}u_1u_3&0&0\\u_3+u_2/u_1&u_2/u_1u_3&0\\1&1/u_3&1/u_2\end{pmatrix}
=\begin{pmatrix}u'_2&0&0\\u'_1&u'_1u'_3/u'_2&0\\1&u'_3/u'_2+1/u'_1&1/u'_1u'_3\end{pmatrix}.
$$
The transition maps are given by
\begin{equation}\label{tm}
u'_1=u_3+u_2/u_1,\qquad u'_2=u_1u_3,\qquad u'_3=u_1u_2/(u_2+u_1u_3).
\end{equation}

Their is a similar parameterisation for $M^{w,e}$, due to Lusztig~\cite{l}.
For $i=1,\ldots,n-1$, set $X_i(v)=I+v f_i$.
Take any reduced decomposition ${\bf i}=i_1\ldots i_r$ for $w$.  
Then the map
$$(v_1,\ldots,v_r) \mapsto X_{i_1}(v_1)\cdots X_{i_r}(v_r)$$
defines a bijection between $\C_{\ne 0}^r$ and $M^{w,e}$.  
This bijection also has the property that the totally positive part $M^{u,v}_{>0}$
of $M^{u,v}$ corresponds precisely to the subset $\R_{>0}^r$ of $\C_{\ne 0}^r$.

In the case $n=3$,  the two representations of an element in $M^{w_0,e}$
corresponding to the words $121$ and $212$, denoting
the corresponding parameters by $(v_1,v_2,v_3)$ and $(v'_1,v'_2,v'_3)$,
respectively, are given by
$$\begin{pmatrix}
1& 0 &0\\
v_1+v_3 & 1 & 0 \\
v_2v_3 & v_2 & 1 \end{pmatrix} 
=\begin{pmatrix}
1& 0 &0\\
v'_2 & 1 & 0 \\
v'_1v'_2 & v'_1+v'_3  & 1 \end{pmatrix} ,$$
with transition maps
\begin{equation}\label{tmv}
v'_1=\frac{v_2 v_3}{ v_1+ v_3},\qquad  v'_2= v_1+ v_3,\qquad  v'_3
=\frac{ v_1 v_2}{ v_1+ v_3}.
\end{equation}

We conclude this section with a simple lemma.
Let $b\in G^{e,w}$ and write $b=an$ where 
$a=\mbox{diag }(a_1,\ldots,a_n)$, say, and $n\in N$.  
Then~\cite{fz}, for any $w\in S_n$, $b \bar w$ has a Gauss (or LDU) 
decomposition $b\bar w = [b \bar w]_- [b \bar w]_0 [b \bar w]_+$ and
$n\bar w$ has a Gauss decomposition $n\bar w = [n \bar w]_- [n \bar w]_0 [n \bar w]_+$.
Moreover, $[n \bar w]_{-0}=[n \bar w]_- [n \bar w]_0 \in L^{w,e}$ and
$[b \bar w]_-\in M^{w,e}$.  Let ${\bf i}=i_1\ldots i_r$ be a reduced decomposition
for $w$.  Then we can write
$$[n \bar w]_{-0}=Y_{i_1}(u_1)\ldots Y_{i_r}(u_r),\qquad
[b \bar w]_- = X_{i_1}(v_1)\ldots X_{i_r}(v_r).$$
Define $Z_i(u)=\varphi_i\begin{pmatrix}u&0\\0&u^{-1}\end{pmatrix}$.
Set $a^0=a$ and, for $1\le k\le r$, $a^{k}=a^{k-1} Z_{i_k}(u_k)$.
Write $a^k=\mbox{diag }(a^k_1,\ldots,a^k_n)$.
\begin{lem}\label{u-v}
The following relations holds: $[b\bar w]_0=a^{r}$ and, for $k=1,\ldots,r$,
$v_k=u_k^{-1} a^{k-1}_{i_k+1}/a^{k-1}_{i_k}$.
\end{lem}
\begin{proof}
Note that $a[n\bar w]_{-0}=[b\bar w]_{-0}= [b \bar w]_- [b \bar w]_0$, hence
$$a Y_{i_1}(u_1)\ldots Y_{i_r}(u_r)= X_{i_1}(v_1)\ldots X_{i_r}(v_r) [b\bar w]_0.$$
The result follows by repeated application of the identity $a Y_i(u)=X_i(v) a'$,
where $a'=aZ_i(u)$ and $v=u^{-1} a_{i+1}/a_i$.
\end{proof}

\section{An evolution on upper triangular matrices}

As shown in \cite{bbo}, the path-transformations $ T_{w}\eta$ can also be 
represented in terms of an evolution on  the upper triangular matrices 
in $GL(n,\R)$.  Let $w=s_{i_1}\cdots s_{i_r}$ be a reduced decomposition and
$\eta:(0,\infty)\to\R^n$ a continuous path.
Set $\eta_ 0=\eta$ and, for $k\le  r$,
\begin{equation}\label{xdef}
\eta_{k}=  T_{i_{k}}\ldots   T_{i_{ 1}}\eta \qquad
x_k(t)=\log\int_0^t e^{-\alpha_{i_k}(\eta_{k-1}(s))} ds.
\end{equation}
Then $\eta_r=  T_{ w}\eta$ and, for each $k\le  r$,
$ \eta_k=\eta+\sum_{j=1}^k x_j \alpha_{i_j} $.

Write $\eta(t)=(\eta^1_t,\ldots,\eta^n_t)$.
Define a path $b(t)$ taking values in $B$ by
$$b_{ij}(t)=e^{\eta^i(t)}
\int_{0<s_{j-1}<s_{j-2}\cdots<s_i<t} 
\exp\left(-\sum_{k=i}^{j-1} \alpha_k(\eta(s_k))\right) ds_i\cdots ds_{j-1}.$$
If $\eta$ is smooth, the $b$ satisfies the ordinary differential equation
$$db = \left(\sum_{i=1}^n h_id\eta^i + \sum_{i=1}^{n-1} e_i dt \right) b,
\qquad b(0)=I.$$
If $\eta$ is a Brownian path (as in the next section)
then $b$ satisfies the equation interpreted as a Stranovich SDE.

When $n=2$,
$$d b = \left( \begin{array}{cc}
d\eta^1 & dt  \\
0 & d\eta^2   \end{array} \right) b,\qquad
b(t)=\left( \begin{array}{cc}
e^{\eta^1_t} & \int_0^t e^{\eta^2_s-\eta^1_s+\eta^1_t} ds \\
0 & e^{\eta^2_t}  \end{array} \right) .$$
When $n=3$,
$$d b = \left( \begin{array}{ccc}
d\eta^1 & dt &0 \\
0 & d\eta^2 & dt\\
0&0&d\eta^3 \end{array} \right) b,$$
and the solution is given by
$$b(t)=\left( \begin{array}{ccc}
e^{\eta^1_t} & \int_0^t e^{\eta^2_s-\eta^1_s+\eta^1_t} ds 
& \int\int_{0<r<s<t} e^{\eta^3_r-\eta^2_r+\eta^2_s-\eta^1_s+\eta^1_t} dr ds\\
0 & e^{\eta^2_t} & \int_0^t e^{\eta^3_s-\eta^2_s+\eta^2_t} ds\\
0 & 0 & e^{\eta^3_t} \end{array} \right) .$$

Write $b=an$, where $a=\mbox{diag}(e^{\eta^1},\ldots,e^{\eta^n})$
and $n\in N$.  Set $u_k=e^{x_k}$ and $v_k=e^{-x_k-\alpha_k(\eta_{k-1})}$.
\begin{thm}\cite{bbo,bbo09}
For each $t>0$, $b(t)\bar w$ has a Gauss decomposition 
$b \bar  w=[b \bar  w]_{-}[b \bar  w]_0[b \bar  w]_{+}$,
with $[b \bar  w]_0=\exp( T_{ w}\eta(t))$.  
Moreover, $[n \bar  w]_{-0}=Y_{i_1}(u_1)\cdots Y_{i_r}(u_r)\in L^{w,e}_{>0} $.
\end{thm}
By Lemma \ref{u-v}, we also have 
$[b \bar  w]_{-}=X_{i_1}(v_1)\cdots X_{i_r}(v_r)\in M^{w,e}_{>0} $.

\subsection{The case $n=2$}
From the definitions: $\alpha_1=e_1-e_2$, $w_0=s_1=s_{e_1-e_2}$,
$$u:=u_1=e^{x_1}=\int_0^t e^{-\eta^1_s+\eta^2_s} ds\qquad
v:=v_1=e^{-y_1}=e^{-\eta^1+\eta^2}u^{-1}$$
$$e^{T_{w_0}\eta} = (e^{\eta^1}u,e^{\eta^2}u^{-1} )
= \left(\int_0^t e^{\eta^2_s+\eta^1_t-\eta^1_s} ds,\ \int_0^t e^{-(\eta^1_s+\eta^2_t-\eta^2_s)}
\right) $$
$$b=
\begin{pmatrix}
e^{\eta^1} & \int_0^t e^{\eta^2_s-\eta^1_s+\eta^1_t} ds \\
0 & e^{\eta^2} \end{pmatrix}
=\begin{pmatrix}
e^{\eta^1} & e^{\eta^1} u  \\
0 & e^{\eta^2}  \end{pmatrix}
= \begin{pmatrix}
e^{\eta^1} & 0  \\
0 & e^{\eta^2}   \end{pmatrix}
\begin{pmatrix}
1 & u  \\
0 & 1  \end{pmatrix} = an
$$
Taking $\bar w_0=\begin{pmatrix}0&-1\\1&0\end{pmatrix}$, we see that
$$b\bar w_0 =\begin{pmatrix}
e^{\eta^1} u & - e^{\eta^1}    \\
e^{\eta^2}&0  \end{pmatrix} 
= \begin{pmatrix}1&0\\v&1\end{pmatrix}
\begin{pmatrix}e^{\eta^1}u&0\\0&e^{\eta^2}u^{-1}\end{pmatrix}
\begin{pmatrix}1&-u^{-1}\\0&1\end{pmatrix}
$$
and
$$n\bar w_0 = \begin{pmatrix}
1 & u  \\
0 & 1  \end{pmatrix}
\begin{pmatrix}0&-1\\1&0\end{pmatrix}
=\begin{pmatrix}u&-1\\1&0\end{pmatrix}
=\begin{pmatrix}u&0\\1&u^{-1}\end{pmatrix}
\begin{pmatrix}1&-u^{-1}\\0&1\end{pmatrix}
$$
Hence 
$$[b\bar w_0]_0 = e^{T_{w_0}\eta},\qquad
[b\bar w_0]_{-}=\begin{pmatrix}1&0\\v&1\end{pmatrix}=X_1(v),
\qquad [n\bar w_0]_{-0}=\begin{pmatrix}u&0\\1&u^{-1}\end{pmatrix}=Y_1(u)$$
as claimed.

\subsection{The case $n=3$}
From the definitions: 
$$\alpha_1=e_1-e_2,\quad \alpha_2=e_2-e_3,\qquad w_0=s_1s_2s_1=s_2s_1s_2.$$
For the reduced decomposition $w_0=s_1s_2s_1$, we have
$$u_1=e^{x_1}=\int_0^t e^{-\eta^1(s)+\eta^2(s)} ds,\qquad
e^{\eta_1}=(e^{\eta^1}u_1,e^{\eta^2}/u_1,e^{\eta^3} )$$
$$u_2=e^{x_2}=\int_0^t e^{-\eta_1^2(s)+\eta_1^3(s)} ds,\qquad
e^{\eta_2}=(e^{\eta^1}u_1,e^{\eta^2}u_2/u_1,e^{\eta^3}/u_2 )$$
$$u_3=e^{x_3}=\int_0^t e^{-\eta_2^1(s)+\eta_2^2(s)} ds,\qquad
e^{\eta_3}=e^{T_{w_0}\eta} = (e^{\eta^1}u_1u_3,e^{\eta^2}u_2/u_1u_3,
e^{\eta^3}/u_2 ).$$
$$v_1=e^{-y_1}=e^{-\eta^1+\eta^2}/u_1\qquad
v_2=e^{-y_2}=e^{-\eta^2+\eta^3}u_1/u_2\qquad
v_3=e^{-y_3}=e^{-\eta^1+\eta^2}u_2/u_1^2 u_3$$

\begin{align*}
b &=\begin{pmatrix}
e^{\eta^1} & \int_0^t e^{\eta^2_s-\eta^1_s+\eta^1_t} ds 
& \int\int_{0<r<s<t} e^{\eta^3_r-\eta^2_r+\eta^2_s-\eta^1_s+\eta^1_t} dr ds\\
0 & e^{\eta^2} & \int_0^t e^{\eta^3_s-\eta^2_s+\eta^2_t} ds\\
0 & 0 & e^{\eta^3}\end{pmatrix} \\
&= \begin{pmatrix}
e^{\eta^1} & 0 &0\\
0 & e^{\eta^2} & 0\\
0 & 0 & e^{\eta^3}\end{pmatrix}
\begin{pmatrix}
1 & u_1 &u_1 u_3\\
0 & 1 & u_3+u_2/u_1\\
0 & 0 & 1\end{pmatrix} =an .
\end{align*}
The identity 
$$\int_0^t e^{\eta^3_s-\eta^2_s+\eta^2_t} ds = u_3+u_2/u_1$$
follows from \eqref{tm}.
Now,
$$
\bar  w_0=  \begin{pmatrix}
0& 0 &1\\
0 & -1 & 0 \\
1 & 0 & 0 \end{pmatrix}  ,$$
so we have
\begin{align*}
&\qquad b\bar w_0 =\begin{pmatrix}
e^{\eta^1} u_1u_3 & - e^{\eta^1}u_1 &  e^{\eta^1}   \\
e^{\eta^2}(u_3+u_2/u_1) &-e^{\eta^2}&0\\
e^{\eta^3}&0&0  \end{pmatrix} \\
&= 
\begin{pmatrix}1& 0 &0\\v_1+v_3 & 1 & 0 \\v_2v_3 & v_2 & 1 \end{pmatrix} 
\begin{pmatrix}
e^{\eta^1}u_1u_3&0&0\\
0&e^{\eta^2}u_2/u_1u_3&0\\
0&0&e^{\eta^3}/u_2 \end{pmatrix}
\begin{pmatrix}
1&-1/u_3&1/u_1u_3 \\
0&1&-u_3/u_2-1/u_1\\
0&0&1\end{pmatrix} 
\end{align*}
and
\begin{align*}
n \bar w_0 &= 
\begin{pmatrix}u_1u_3 & - u_1 &  1   \\u_3+u_2/u_1 &-1&0\\1&0&0  \end{pmatrix} \\
&= 
\begin{pmatrix}u_1u_3&0&0\\u_3+u_2/u_1&u_2/u_1u_3&0\\1&1/u_3&1/u_2\end{pmatrix}
\begin{pmatrix}1&-1/u_3&1/u_1u_3 \\ 0&1&-u_3/u_2-1/u_1\\ 0&0&1\end{pmatrix} 
\end{align*}
Thus $[b\bar w_0]_0 = e^{T_{w_0}\eta}$,
$$[b\bar w_0]_{-}=\begin{pmatrix}1& 0 &0\\v_1+v_3 & 1 & 0 \\v_2v_3 & v_2 & 1 \end{pmatrix} 
=X_1(v_1)X_2(v_2)X_3(v_3),$$
$$ [n\bar w_0]_{-0}=
\begin{pmatrix}u_1u_3&0&0\\u_3+u_2/u_1&u_2/u_1u_3&0\\1&1/u_3&1/u_2\end{pmatrix}
=Y_1(u_1)Y_2(u_2)Y_3(u_3)$$
as claimed.

\subsection{Evolution of the Lusztig parameters}

As before, we introduce a probability measure $\mathbb P$ under which $\eta$ 
is a Brownian motion in $\R^n$ with a drift $\mu$ and $\eta(0)=0$.  
For each $k\le n$, set
$$( T_{k,1},\ldots, T_{k,k})= T^{(k)}(\eta^1,\ldots,\eta^k).$$
Note that this is given in terms of the principal minors $b^{(k)}$, $k\le n$, of $b$ by
$T^{(k)}(\eta^1,\ldots,\eta^k) =\log [b^{(k)}\bar w^{(k)}_0]_0$,
where $w^{(k)}_0$ denotes the longest element  in $S_k$.
The evolution of the triangular array $   T_{k,j},\ 1\le j\le k \le n$
is given by \eqref{T-ev}.  As remarked earlier, this process
contains a number of projections which are also Markov.  
In particular, setting $\xi_k=T_{k,k}$, we have, for $k\le n$,
$$d   \xi_k = d\eta^k - e^{  \xi_k-   \xi_{k-1}} dt .$$
This defines a simple interacting particle system on the real line which, 
in the coordinates $\sum_i\xi_i$ and $\xi_{i+1}-\xi_i$, $1\le i\le n-1$, has 
a product form invariant measure.  There is an extension of this process, 
involving the Lusztig parameters, which is also Markov and, moreover, 
also has a product form invariant measure.
Let $v_1,\ldots,v_q$ be the Lusztig parameters corresponding to a reduced decomposition
$w_0=s_{i_1}\ldots s_{i_q}$, that is,
$$[b\bar w_0]_{-}=X_{i_1}(v_1)\cdots X_{i_{q}}(v_{q}).$$
Set $y_k=-\log v_k$.  The evolution of $y_k$, $1\le k\le q$, is given by
$$dy_k=d\alpha_{i_k}(\eta_{k-1})+e^{-y_k}dt,$$
where $\eta_k=T_{i_k}\ldots T_{i_1}\eta$.
Setting $x_k=y_k-\alpha_{i_k}(\eta_{k-1})$, note that $dx_k=e^{-y_k}dt$
and $\eta_k=\eta+\sum_{j=1}^k x_j\alpha_j$.
Hence,
\begin{equation}\label{l}
dy_k=d\alpha_{i_k}(\eta)+\sum_{j=1}^{k-1}\alpha_{i_k}(\alpha_{i_j})e^{-y_j}dt+e^{-y_k}dt.
\end{equation}
Let $\beta_1=\alpha_{i_1}$ and, for $2\le k\le q$, $\beta_k=s_{i_1}\ldots s_{i_{k-1}}\alpha_{i_k}$.
Set $\theta_k=-\beta_k(\mu)$.  If $\mu\in w_0\Omega=-\Omega$, then $\theta_k>0$ for all $k$ and the 
diffusion has stationary distribution given by the product measure
$$\pi=\bigotimes_{k=1}^q \G(\theta_k)^{-1} g_{\theta_{k}},$$
where $g_{\theta}(dx) = \exp(-\theta x-e^{-x}) dx$.
This can be seen as a consequence of the following fact, which is the analogue in this setting of the
output theorem for the $M/M/1$ queue~\cite{oy2}.  Let $x_t$ be a standard one-dimensional
Brownian motion with negative drift $-\theta$, and consider the one-dimensional diffusion
$$dy=\sqrt{2}dx+e^{-y}dt.$$
This has a unique invariant distribution $\G(\theta)^{-1}g_\theta$.  If we start this diffusion
in equilibrium and define $\tilde x_t=x_t+2(y_0-y_t)$, then $\tilde x$ has the same law as
$x$ and, moreover, $\tilde x_s,\ s\le t$ is independent of $y_u,\ u\ge t$, for all $t$.
It follows that the measure $\pi$ is invariant.  For an analytic proof of this fact, see~\cite{oo}.  
See also~\cite[Proposition 5.9]{bbo09}, where the equivalent property is proved in the `zero-temperature' setting.

If we choose the reduced decomposition
${\bf i}= 1\ 21\ 321\ n-1\ n-2\ldots 21$,
and define,  for $m\le n-1$ and $1\le i\le n-m$,  
$q_{m,i}=T_{i+m,i+1}-   T_{i+m-1,i}$, then 
$$(y_1,y_2,\ldots,y_q)=(q_{1,1},q_{1,2},\ldots,q_{1,n},q_{2,1},\ldots,q_{2,n-1},\ldots,q_{n-1,1}).$$
Note that $q_{1,i}=\xi_{i+1}-\xi_{i}$, for $1\le i\le n-1$.  In these coordinates, the
evolution is given by
$$dq_{m,i}=d\alpha_i(\eta)+e^{-q_{m,i}}dt
+\sum_{l=1}^{m-1} \left( 2e^{-q_{l,i}}-e^{-q_{l,i+1}}-e^{-q_{l,i-1}}\right) dt,$$
with the conventions that the empty sum is zero and $q_{l,0}=+\infty$.
Setting $\theta_{m,i}=\mu_{m+i}-\mu_m$, an invariant measure for this diffusion
is given by the product measure $\bigotimes_{m,i} g_{\theta_{m,i}}$.
The dynamics of this process can be viewed as a network, as follows.
Consider the dynamics
$$dQ=d(A-S)+e^{-Q}dt,\qquad dD=dA-dQ,\qquad dT=dS+dQ.$$
We think of $A,S$ as the input and $D,T$ as the output, and represent 
this system graphically as:
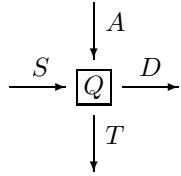
\setlength{\unitlength}{1.5cm}
\setlength{\fboxsep}{.2cm}

\begin{center}

\begin{picture}(4,1.7)(-2,1.2)

\put(0,1.85){\framebox(.3,.3){$Q$}}

\put(-.4,2.1){$S$}
\put(.25,2.5){$A$}

\put(.55,2.1){$D$}
\put(.25,1.5){$T$}

\put(-.6,2){\vector(1,0){.5}}
\put(.4,2){\vector(1,0){.5}}

\put(0.15,2.75){\vector(0,-1){.5}}
\put(0.15,1.75){\vector(0,-1){.5}}

\end{picture}
\end{center}

Then the evolution of the $q_{m,i}$ can be represented as in Figure 2.
To see directly from this picture the product-form invariant measure, note
that, if $A$ and $S$ are independent standard one-dimensional Brownian motions
with respective drifts $\lambda$ and $\sigma$, with $\lambda<\sigma$, 
then the diffusion $Q$ has invariant distribution $\G(\theta)^{-1}g_\theta$,
where $\theta=\sigma-\lambda$.  Moreover, if we start this diffusion in
equilibrium, then $D_t=A_t+Q_0-Q_t$ and $T_t=S_t-Q_0+Q_t$ are
independent standard one-dimensional Brownian motions
with respective drifts $\lambda$ and $\sigma$, and for each $t>0$,
$(D_s,T_s),\ s\le t$ is independent of $Q_u,\ u\ge t$.  The analogue
of this fact in the setting of Poissonian queueing networks  is the
cornerstone of classical queueing theory.  It is called the output, or Burke, theorem.
Finally, we remark that the dynamics indicated by Figure 2 is the analogue,
in this setting, of the dynamical interpretation given in \cite{noc-rsk} of the 
RSK correspondence as a kind of `queueing network'.

\begin{figure}[h]

\setlength{\unitlength}{2.5cm}
\setlength{\fboxsep}{.2cm}

\begin{center}

\begin{picture}(5,3.5)(-.7,-.7)

\put(2,-.15){\framebox(.3,.3){$q_{3,1}$}}
\put(1,.85){\framebox(.3,.3){$q_{2,1}$}}
\put(2,.85){\framebox(.3,.3){$q_{2,2}$}}
\put(0,1.85){\framebox(.3,.3){$q_{1,1}$}}
\put(1,1.85){\framebox(.3,.3){$q_{1,2}$}}
\put(2,1.85){\framebox(.3,.3){$q_{1,3}$}}

\put(-.4,2.1){$\eta^1$}
\put(.25,2.5){$\eta^2$}
\put(1.25,2.5){$\eta^3$}
\put(2.25,2.5){$\eta^4$}

\put(.55,2.1){$T_{2,2}$}
\put(1.55,2.1){$T_{3,3}$}
\put(2.55,2.1){$T_{4,4}$}

\put(.3,1.1){$T_{2,1}$}
\put(1.55,1.1){$T_{3,2}$}
\put(2.55,1.1){$T_{4,3}$}

\put(1.3,.1){$T_{3,1}$}
\put(2.55,.1){$T_{4,2}$}

\put(2.25,-.55){$T_{4,1}$}

\put(-.6,2){\vector(1,0){.5}}
\put(.4,2){\vector(1,0){.5}}
\put(1.4,2){\vector(1,0){.5}}
\put(2.4,2){\vector(1,0){.5}}
\put(1.4,1){\vector(1,0){.5}}
\put(2.4,1){\vector(1,0){.5}}
\put(2.4,0){\vector(1,0){.5}}

\put(0.15,2.75){\vector(0,-1){.5}}
\put(1.15,2.75){\vector(0,-1){.5}}
\put(2.15,2.75){\vector(0,-1){.5}}

\put(0.15,1.75){\line(0,-1){.75}}
\put(0.15,1){\vector(1,0){.75}}

\put(1.15,.75){\line(0,-1){.75}}
\put(1.15,0){\vector(1,0){.75}}

\put(1.15,1.75){\vector(0,-1){.5}}
\put(2.15,1.75){\vector(0,-1){.5}}

\put(2.15,.75){\vector(0,-1){.5}}
\put(2.15,-.25){\vector(0,-1){.5}}

\end{picture}
\end{center}

\caption{Graphical representation of the evolution of Lusztig parameters}

\end{figure}
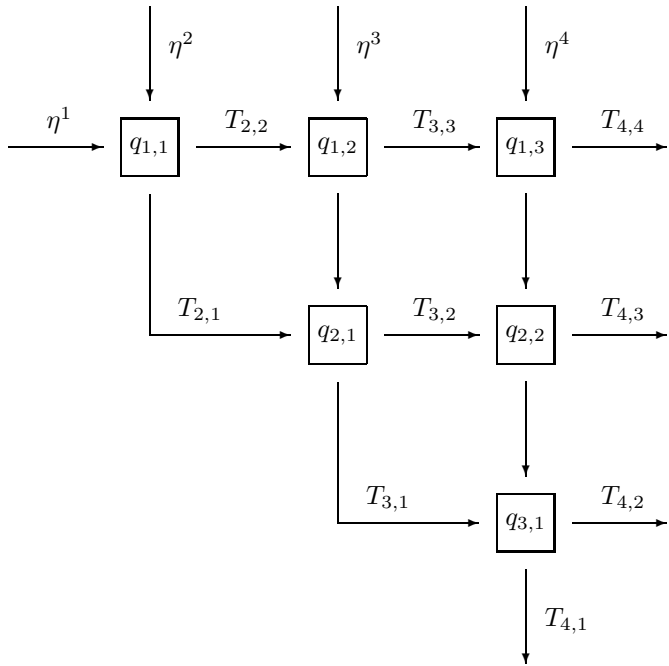

\section{From the Feyman-Kac formula to Givental's integral formula}

The fact that the evolution equation \eqref{l} for the Lusztig parameters
has a product form invariant measure sheds some light on the relation
between the Feyman-Kac formula~\eqref{fk} and the integral formula
of Givental.  It follows from this that, for
any given reduced decomposition of $w_0$, the random variables
$$\int_0^\infty e^{-\alpha_i(\beta_s)} ds,\qquad i=1,\ldots,n-1$$
can be expressed, via the transition maps, as rational functions of a collection 
of $q=n(n-1)/2$ independent Gamma-distributed random variables with 
respective parameters $\theta_k$, $k\le q$, defined as above with $\beta=-\eta$.  
Note that $\beta$ is a Brownian motion with drift $\lambda=-\mu\in\Omega$. 
Since the sets
$\{\theta_k,\ k\le q\}$ and $\{\l_i-\l_j,\ i<j\}$ are the same, this allows \eqref{fk}
to be written as a $q$-dimensional integral
\begin{align}\label{nif}
\psi_\lambda(x)&=\prod_{i<j} \G(\lambda_i-\lambda_j) e^{\lambda(x)}
\E_x \exp\left(-\sum_{i=1}^{n-1}\int_0^\infty e^{-\alpha_i(\beta_s)} ds\right) \nonumber\\
&= e^{\lambda(x)} \int_{\R_+^q} e^{-\sum_{i=1}^{n-1} e^{-\alpha_i(x)} r_i(v_1,\ldots,v_q)} 
\prod_{i=1}^q v_i^{\theta_i-1} e^{-v_i} dv_i .
\end{align}
For example, when $n=3$ and $\mathbf{i}=121$, we have 
$$\theta_1=\lambda_1-\lambda_2,\quad \theta_2=\lambda_1-\lambda_3, \quad \theta_3=\lambda_2-\lambda_3,$$
and, using (\ref{tmv}),
$$r_1(v_1,v_2,v_3)=\frac1{v_1},\qquad r_2(v_1,v_2,v_3)=\frac1{v'_1}=\frac{v_1+v_3}{v_2 v_3}.$$
In this case, the integral formula (\ref{nif}) becomes
\begin{align*}
\psi_\lambda(x) = e^{\lambda_1 x_1+\lambda_2 x_2+\lambda_3 x_3} 
 \int_{\R_+^3} & v_1^{\lambda_1-\lambda_2-1} v_2^{\lambda_1-\lambda_3-1} v_3^{\lambda_2-\lambda_3-1} \\
& \times \exp\left(-v_1-v_2-v_3-e^{-x_1+x_2} \frac1{v_1} -e^{-x_2+x_3} \frac{v_1+v_3}{v_2 v_3} \right)
dv_1dv_2dv_3 .
\end{align*}
Under the change of variables 
$$v_1=e^{T_{32}-T_{21}},\qquad v_2=e^{T_{33}-T_{22}},\qquad v_3=e^{T_{22}-T_{11}},$$
where $T=(T_{ki},\ 1\le i\le k\le 3)$ is an array with $(T_{31},T_{32},T_{33})=(x_1,x_2,x_3)$, this integral becomes
\begin{align*}
\psi_\lambda(x) = 
 \int_{\R^3} & e^{\lambda_1(T_{31}+T_{32}+T_{33}-T_{21}-T_{11})+\lambda_2 (T_{21}+T_{22}-T_{11})+\lambda_3 T_{11}}\\
& \times \exp\left(-e^{T_{32}-T_{21}}-e^{T_{33}-T_{22}}-e^{T_{22}-T_{11}}
 -e^{T_{21}-T_{31}}-e^{T_{11}-T_{22}}-e^{T_{22}-T_{32}} \right)
dT_{11} dT_{21} dT_{22}.
\end{align*}

Since $\Psi_\lambda(x)$ is a symmetric function of $\lambda$ we see that this agrees with Givental's integral formula (\ref{g}).
We note that this is reminiscent of the derivation of Givental's formula given in~\cite{GLO}
(see also \cite{gklo,GKMMMO}).

\section{Fundamental Whittaker functions}

The eigenvalue equation \eqref{ee} also has series solutions
known as {\em fundamental} Whittaker functions.  
Define a collection of analytic functions 
$a_{n,m}(\nu)$, $n\ge 2$, $m\in (\Z_+)^{n-1}$, $\nu\in\C^n$ recursively by
$$a_{2,m}(\nu)=\frac1{m!\G(\nu_1-\nu_2+m+1)},$$
and for $n>2$,
$$a_{n,m}(\nu)=\sum_{k} a_{n-1,k}(\mu) \prod_{i=1}^{n-1} \frac1{(m_i-k_i)!}
\frac1{\G(\nu_i-\nu_n+m_i-k_{i-1})} ,$$
where $\mu_i=\nu_i+\nu_n/(n-1)$, $i\le n-1$, and the sum is over $k\in (\Z_+)^{n-2}$
satisfying $k_i\le m_i$, $1\le i\le n-2$, with the convention that $k_0=k_{n-1}=0$.
Then~\cite[Theorem 15]{is} for each $n$, $a_{n,m}(\nu)$ satisfies the recursion
$$ \left[\sum_{i=1}^{n-1} m_i^2 - \sum_{i=1}^{n-2}m_i m_{i+1} 
+\sum_{i=1}^{n-1} (\nu_i-\nu_{i+1}) m_i \right] 
a_{n,m}(\nu) =  \sum_{i=1}^{n-1} a_{n,m-e_i}(\nu),$$
with the convention that $a_{n,m}=0$ for $m\notin (\Z_+)^{n-1}$, and
$a_{n,0}(\nu)=\prod_{i<j}\Gamma(\nu_i-\nu_j+1)^{-1}$.
Writing $m'=\sum_{i=1}^{n-1} m_i (e_i-e_{i+1})$, the series
\begin{equation}\label{series}
m_\nu(x)=\sum_m a_{n,m}(\nu) e^{-(m'+\nu,x)}
\end{equation}
is a fundamental Whittaker function as defined by Hashizume~\cite{h},
and satisfies the eigenvalue equation \eqref{ee}.  We adopt a slightly different
normalisation than the ones used in the papers~\cite{h} or \cite{is}.
Note that, for each $x\in\R^n$, $m_\nu(x)$ is an analytic function of $\nu$.
Moreover:
\begin{prop}
$$\psi_\nu(x) = \prod_{i<j} \frac{\pi}{\sin\pi(\nu_i-\nu_j)} \sum_{ w\in S_n} (-1)^ w
m_{- w \nu} (x).$$
\end{prop}
\begin{proof} This comes from~\cite{boc}.
In the notation of that paper we are taking 
$\Pi= \{\a_i/2,\ i=1,\ldots,n-1\}$, $m(2\alpha)=0$, $|\eta_\alpha|^2=1$ and 
$\psi_\nu(x)=2^q k_\nu(x)$ where $q=n(n-1)/2$.  
\end{proof}

Now consider the function $\theta_t(x)$ defined by \eqref{theta}.
Note that we can write
$$s_n(\l)=\frac1{(2\pi \iota)^n n!} h(\l) \prod_{i>j}\frac{\sin \pi(\l_i-\l_j)}{\pi}.$$
\begin{cor}
$$\theta_t(x) = \frac1{(2\pi \iota)^n } \int_{\iota \R^n} m_\l(x) h(\l) e^{\sum_i \l_i^2 t/2} d\l .$$
\end{cor}

\section{Relativistic Toda and $q$-deformed Whittaker functions}

The algebraic structure underlying Theorem~\ref{main} is an intertwining relation
between certain differential operators associated with the open quantum Toda chain
with $n$ particles. This
structure should carry over to the setting of Ruijsenaars' relativistic Toda difference 
operators and $q$-deformed Whittaker functions~\cite{ru,ru1,et,g-q}.
A recent (related, but different) development along these lines is given in~\cite{bc}.  
We will describe here the $q$-analogue of Theorem~\ref{main} in the rank one case, 
which corresponds to $n=2$.  

In the case $n=2$, the Whittaker function is given by
$$\psi_\l(x) = 2 \exp\left(\frac12(\l_1+\l_2)(x_1+x_2)\right)
K_{\l_1-\l_2}\left( 2 e^{(x_2-x_1)/2}\right),$$
where $K_\nu(z)$ is the Macdonald function.
In this case, Theorem~\ref{main} is equivalent to the
following theorem of Matsumoto and Yor~\cite{my}.
\begin{thm}\label{my}
\begin{enumerate}
\item
Let $(B^{(\mu)}_t, t\ge 0)$ be a Brownian motion with drift $\mu$, 
and define $$Z^{(\mu)}_t=\int_0^t e^{2B^{(\mu)}_s-B^{(\mu)}_t}ds.$$
Then $\log Z^{(\mu)}$ is a diffusion process with infinitesimal generator
$$\frac12 \frac{d^2}{dx^2}+\left(\frac{d}{dx}\log K_\mu(e^{-x})\right) \frac{d}{dx}.$$
\item The conditional law of $B^{(\mu)}_t$, given $\{Z^{(\mu)}_s,\ s\le t; Z^{(\mu)}_t=z\}$,
is given by the generalized inverse Gaussian distribution
$$ \frac12 K_\mu(1/z)^{-1} e^{\mu x} \exp\left( - \cosh(x)/z \right) dx .$$
\end{enumerate}
\end{thm}

Let $0\le q<1$.  Denote the $q$-Pochhammer symbol by 
$(q)_n=(q;q)_n=(1-q)\cdots (1-q^n)$ with the conventions
that $(q)_0=1$ and $(0)_n=1$.  In what follows we also adopt the
convention that $0^0=1$. 

For $\lambda\in\C$ and $z\ge 0$, define
$$\psi_\lambda(z)=\sum_{y=0}^z \frac{q^{\lambda(2y-z)}}{(q)_y(q)_{z-y}}.$$
This is a $q$-deformed Whittaker function associated with 
$\mathfrak{sl}_2$~\cite{g-q}.  It satisfies the difference equation
$$(1-q^{z+1})\psi_\l(z+1)+\psi_\l(z-1)=(q^\l+q^{-\l})\psi_\l(z)$$
where we set $\psi_\l(-1)=0$, and is related to the $q$-Hermite polynomials by
$$ (q)_z \psi_\l(z)= H_z\left(\left.\frac{q^\l+q^{-\l}}2\right| q\right).$$

Fix $0\le q<1$, $0\le p\le 1$ and let $(Y_n,Z_n)_{n\ge 0}$ be a Markov chain with 
state space $\{(y,z)\in\Z^2: z\ge y\ge 0\}$ and transition probabilities given by
$$\Pi((y,z),(y+1,z+1))=p,\qquad
\Pi((y,z),(y,z+1))=(1-p)q^y,$$
$$\Pi((y,z),(y-1,z-1))=(1-p)(1-q^y).$$
Note that $Y$ is itself a Markov chain with transition probabilities
$$P(y,y+1)=p,\quad P(y,y)=(1-p)q^y,\quad P(y,y-1)=(1-p)(1-q^y),$$
and $X=2Y-Z$ is a simple random walk on the integers which 
increases by one with probability $p$ and decreases by one with probability $1-p$.
Choose $\nu\in\R$ such that $p=q^\nu/(q^\nu+q^{-\nu})$.
\begin{thm}\label{q-MY} Let $Y_0=Z_0=0$.
The process $(Z_n,n\ge 0)$ is a Markov chain with transition probabilities
$$Q(z,z+1)=\frac{1-q^{z+1}}{q^\nu+q^{-\nu}}\frac{\psi_\nu(z+1)}{\psi_\nu(z)},\qquad
Q(z,z-1)=\frac{1}{q^\nu+q^{-\nu}}\frac{\psi_\nu(z-1)}{\psi_\nu(z)}.$$
Moreover, for each $n\ge 0$, the conditional distribution of $Y_n$, given $\sigma\{Z_m,m\le n\}$
and $Z_n=z$, is given by
$$\pi_z(y)=\psi_\nu(z)^{-1}  \frac{q^{\nu(2y-z)}}{(q)_y(q)_{z-y}} ,\qquad y=0,1,\ldots,z.$$
\end{thm}
The proof is straightforward using the theory of Markov functions, by which it suffices to check
that the transition operators $\Pi$ and $Q$ satisfy the intertwining relation $QK=K\Pi$ where
$$K(z,(y,z'))= \frac{\delta_{z,z'}q^{\nu(2y-z)}}{\psi_\nu(z)(q)_y(q)_{z-y}}.$$
This intertwining relation is readily verified.
When $q=0$ and $\nu=0$, $\psi_\nu(z)=z$ and the above theorem can be interpreted
as the discrete version of Pitman's `$2M-X$' theorem, which states that if $X_n$ is a
simple symmetric random walk and $M_n=\max_{m\le n} X_m$, then $2M-X$ is a
Markov chain with transition probabilities
$Q(z,z+1)=(z+1)/2z,\ Q(z,z-1)=(z-1)/2z$.  When $q\to 1$, it should rescale to Theorem~\ref{my}.

The analogue of the output/Burke theorem in the setting of Theorem~\ref{q-MY} 
is the following.  If $p<1/2$,
then the Markov chain $Y$ has a stationary distribution.  If $Y_0$ is chosen according
to this distribution and $Z_0=0$, the process $(Z_n,n\ge 0)$ is a simple random walk 
on the integers which increases by one with probability $p$ and decreases by one 
with probability $1-p$.

\bigskip

\noindent {\bf Acknowledgements.}  Many thanks to Nikos Zygouras and an anonymous
referee for careful reading of the manuscript and for valuable suggestions which 
lead to a much improved version.  Thanks also to Alexei Borodin for helpful 
discussions regarding Proposition 3.  
This review started out as draft lecture notes for the 
summer school {\em Random matrix theory and applications} held at the Indian Institute 
of Science, Bangalore, in January, 2012.  Unfortunately, I didn't make it to Bangalore 
(for bureaucratic reasons) but I am very grateful to the organisers for all their efforts.  
Research supported in part by EPSRC grant EP/I014829/1.

\end{document}